\documentclass[12pt,reqno]{amsart}

\usepackage[margin=1in]{geometry}
\usepackage{color}
\usepackage[mathscr]{euscript}
\usepackage{graphicx,caption}
\usepackage{hyperref}
\usepackage{cite}

\newcommand{\cL}{\mathcal{L}}
\newcommand{\cM}{\mathcal{M}}

\newcommand{\cG}{\mathcal{G}}

\newcommand{\cX}{\mathcal{X}}
\newcommand{\cY}{\mathcal{Y}}

\newcommand\R{\mathbb{R}}
\newcommand\C{\mathbb{C}}

\newcommand\term[1]{\emph{#1}}

\newcommand{\Lframe}[2]{\big( #1, #2 \big) }

\newcommand{\form}{\mathfrak q}
\DeclareMathOperator{\Mas}{Mas}
\newcommand{\Qform}{\Omega}
\newcommand{\Qmat}{W}

\DeclareMathOperator\dom{dom}
\DeclareMathOperator\mul{mul}
\DeclareMathOperator\Ran{Ran}

\DeclareMathOperator\Sig{Sig}
\DeclareMathOperator\iD{\iota}
\newcommand\cbd{\delta} 

\newcommand{\Ver}{{\cG_\infty}}
\newcommand{\Hor}{{\cG_0}}

\newtheorem{theorem}{Theorem}[section]	
\newtheorem{corollary}[theorem]{Corollary}
\newtheorem{lemma}[theorem]{Lemma}
\newtheorem{proposition}[theorem]{Proposition}

\theoremstyle{definition}

\theoremstyle{remark}
\newtheorem{remark}[theorem]{Remark}
\newtheorem{example}[theorem]{Example}

\numberwithin{equation}{section}

\begin{document}

\title[Duistermaat's triple index]{The calculus of Duistermaat's triple index}

\author[G.\ Berkolaiko]{Gregory Berkolaiko}
\address{Department of Mathematics,
	Texas A\&M University, College Station,
	TX 77843, USA}
\email{gberkolaiko@tamu.edu}

\author[G.\ Cox]{Graham Cox}
\address{Department of Mathematics and Statistics, Memorial University of Newfoundland, St. John's, NL A1C 5S7, Canada}
\email{gcox@mun.ca}

\author[Y.\ Latushkin]{Yuri Latushkin}
\address{Department of Mathematics,
	The University of Missouri, Columbia, MO 65211, USA}
\email{latushkiny@missouri.edu}

\author[S.\ Sukhtaiev]{Selim Sukhtaiev}
\address{Department of Mathematics and Statistics,
	Auburn University, Auburn, AL 36849, USA}
\email{szs0266@auburn.edu}

\begin{abstract}
  In this paper we develop a systematic calculus for the Duistermaat
  index, a symplectic invariant defined for triples of Lagrangian
  subspaces. Introduced nearly half a century ago, this index has
  lately been the subject of renewed attention, due to its central
  role in eigenvalue interlacing problems on quantum graphs (and more
  abstractly for self-adjoint extensions of symmetric operators). Here
  we give an axiomatic characterization of the index that leads to
  elementary proofs of its fundamental properties. We also relate the
  index to other quantities often appearing in symplectic geometry,
  such as the H\"ormander--Kashiwara--Wall index and the Maslov
  index. Among other things, this leads to a curious formula for the
  Morse index of a difference of Hermitian matrices.
\end{abstract}

\maketitle


\section{Introduction}

The relative positions of eigenvalues of two self-adjoint operators that differ by a finite-rank perturbation in their boundary conditions is a question arising in many applications.  These include the TE/TM mode
analysis in graph-like photonic structures \cite{Kuc_incol01},
stability analysis of brake orbits in Lagrangian systems
\cite{HuWuYan_jdde20,HuPorWuXin_n26}, analysis of nodal and Robin
domains \cite{Ber_cmp08,AloBanBer_cmp17,BanProSof_prep25} and spectral
minimal partitions \cite{BanBerRazSmi_cmp12,HofKen_lmp21} on quantum
graphs.
From an abstract perspective, tracking eigenvalue changes due to
variations in boundary conditions is equivalent to comparing
self-adjoint extensions of a symmetric operator with finite and equal
defect numbers. Such spectral comparison problems for differential
operators may be approached via topological indices \cite{Arn_fap85,Bot_cpam56,CZ84,Dui_am76}.  In particular, the Maslov
index \cite{Arn_faa67,Ke58,CapLeeMil_cpam94,
  deGosson_PrinciplesNewtonianQuantum,Leray_LagrangianAnalysis,Piccione,RobSal_t93}
has played a major role in these investigations, and is an active area
of research; see \cite{Beck} for a review and
\cite{AgrBarBes_n23,BanProSof_prep25,BarOffPorWu_mz21,
  Beck_2018,BJ22,BerCoxLatSuk_prep23, 
  CDB09,CDB11,GarPorWu_prep25,How_jmaa21, HLS17, HLS18,HS16,
  Howard_22, LS20} for a partial list of papers specifically focused on the finite
defect setting.
The main theme in these papers is to equate the Maslov index to the
eigenvalue counting function of some linear differential operator,
which is usually obtained by linearizing a dynamical system at an
equilibrium solution. However, the Maslov index is difficult to
compute in practice.  It depends on a path of Lagrangian subspaces, and so without complete knowledge of this path (which amounts to having an explicit solution for the linearized system) it cannot be computed effectively.

 

The fundamental observation of
\cite{HuWuYan_jdde20,BerCoxLatSuk_prep23,GarPorWu_prep25} is that the
\emph{difference} of two eigenvalue counting functions can be evaluated
in terms of a much simpler index that only depends on the endpoints of
the path.
%
%
%
The so-called \emph{Duistermaat index}, introduced in \cite{Dui_am76},
takes three Lagrangian subspaces as its input. In the metric graph
setting these correspond to the two sets of vertex conditions under
consideration and also the Dirichlet conditions at the affected
vertices.  Duistermaat's definition
is somewhat indirect, however, as it involves an arbitrary choice of a fourth
Lagrangian plane that is transversal to the three planes of
interest. The proof that the resulting index is independent of this
choice relies on the Maslov index for paths of Lagrangian subspaces,
see \cite[Eq.~(2.16)]{Dui_am76}. 
As noted above, the Maslov index is inherently more complicated than an index depending on
\emph{fixed} Lagrangian subspaces.

The principal aim of this note is to define the Duistermaat index directly, 
without recourse to the Maslov index.  We do
this by giving a set of axioms (inspired by the approach of
\cite{CapLeeMil_cpam94}) as well as a
formula requiring nothing more than linear algebra.  It is the set of
axioms that turns out to be most useful in practice, allowing us to
give short and simple proofs for many known and some unknown
properties of the Duistermaat index.

We consider $\C^n \oplus \C^n$ as a Hermitian symplectic space with
the canonical symplectic form
\begin{equation}
  \label{eq:sympl_form}
  \omega\big((x_1,x_2), (y_1,y_2) \big) := \left<x_1, y_2\right>_{\C^n} - \left<x_2, y_1\right>_{\C^n}.
\end{equation}
A subspace $\cL \subset \C^n \oplus \C^n$ is
\emph{Lagrangian} if it coincides with its
symplectic complement, 
\begin{equation}
  \label{eq:sympl_complement}
  \cL^\omega := \left\{v \in \C^n \oplus \C^n \colon \omega(u, v) = 0
    \text{ for all } u \in \cL \right\}.
\end{equation}
The set of Lagrangian subspaces (often
called \emph{Lagrangian planes}) in 
$(\C^n \oplus \C^n , \omega)$ is denoted by $\Lambda(n)$.  Every
Lagrangian subspace has dimension $n$, half of the dimension of the
symplectic space, therefore
$\cL_1$ and $\cL_2$ are \emph{transversal} whenever
$\cL_1 \cap \cL_2 = \{0\}$.

An important example of a Lagrangian plane is the graph of
a Hermitian matrix $M$,
\begin{equation}
  \label{eq:graph_def}
  \cG_M := \big\{
    (x, Mx)
    : x\in \C^n \big\}.
\end{equation}
In particular, the ``horizontal'' plane,
\begin{equation}
  \label{eq:hor_def}
  \Hor := \C^n\oplus0,  
\end{equation}
will play an important role in this note.  The
``vertical'' plane $0\oplus\C^n$ is 
not the graph of a Hermitian matrix,
even though we will abuse the notation in \eqref{eq:graph_def} to denote
it
\begin{equation}
  \label{eq:ver_def}
  \Ver := 0\oplus\C^n.  
\end{equation}
Being ``somewhat vertical'' is the only obstacle for a
Lagrangian plane to be a graph: a Lagrangian plane $\cL$ is
the graph of a Hermitian matrix if and only if $\cL$ is transversal to
$\Ver$.

The \emph{Duistermaat index}
is an integer
symplectic invariant of a triple of Lagrangian planes, describing their
relative position.  
Intuitively, it can be thought of as the maximal
dimension of a subspace of $\cL_3$ that lies between $\cL_1$ and
$\cL_2$.  The formal definition we propose differs from that of
\cite{Dui_am76} by taking an axiomatic approach.

\begin{theorem}
  \label{def:Dui_index}
  There exists a unique function
  $\iD : \Lambda(n)^3 \to \{0,1,\ldots,n\}$, called the
  \emph{Duistermaat index}, that satisfies the following properties.
  \begin{enumerate}
  \item \label{item:norm} (Normalization) For any Hermitian matrix
    $A$,
    \begin{equation}
      \label{eq:normalization}
      \iD(\Hor, \cG_A, \Ver) = n_-(A),
    \end{equation}
    where $n_-(A)$ denotes the number of negative eigenvalues (also
    known as the \emph{Morse index}) of the matrix $A$.
  \item \label{item:sympl_inv}
    (Symplectic Invariance) For any symplectic transformation $\mathsf{S}$ on
    $\C^n \oplus \C^n$, 
    \begin{equation}
      \label{eq:sympl_inv}
      \iD(\mathsf{S}\cL_1, \mathsf{S}\cL_2, \mathsf{S}\cL_3) = \iD(\cL_1, \cL_2, \cL_3).
    \end{equation}
  \item \label{item:cocycle_prop}
    (Cocycle Property) For any Lagrangian planes $\cL_1$, $\cL_2$,
    $\cL_3$ and $\cL_4$,
    \begin{equation}
      \label{eq:cocyle_identity}
      \iD(\cL_1,\cL_2,\cL_3) - \iD(\cL_1,\cL_2,\cL_4)
      + \iD(\cL_1,\cL_3,\cL_4) - \iD(\cL_2,\cL_3,\cL_4) = 0.
    \end{equation}
  \end{enumerate}
\end{theorem}

\begin{remark}
  \label{rem:comparison_def}
  The Duistermaat index as defined in \cite{Dui_am76} and \cite{ZhoWuZhu_fmc18} satisfies the axioms in
  Theorem~\ref{def:Dui_index} and therefore coincides with our version.
\end{remark}

\begin{remark}
  \label{rem:cocycle}
  To motivate the terminology in \eqref{item:cocycle_prop}, we recall the coboundary operator $\cbd$ (see \cite[Chapter
  3]{Hatcher}) which maps functions of $k-1$ arguments to functions of
  $k$ arguments by
  \begin{equation}
    \label{eq:coboundary_def}
    \big(\cbd \phi\big) (\ell_1, \ldots, \ell_k)
    = \sum_{j=1}^k (-1)^{k-j} \phi(\ell_1,\ldots,\widehat{\ell}_j, \ldots, \ell_k),
  \end{equation}
  where a hat indicates the argument has been omitted.  Then 
  \eqref{eq:cocyle_identity} says that the coboundary of $\iD$ is
  zero, i.e., $\iD$ is a cocycle.
\hfill$\Diamond$\end{remark}


Theorem~\ref{def:Dui_index} will be established in
Section~\ref{sec:existence} by supplying an explicit and self-contained, if somewhat
cumbersome, formula for the index.  A more practical formula is the
following extension of \eqref{eq:normalization}.

\begin{theorem}
  \label{thm:general_formula}
  For any $n \times n$ Hermitian matrices $A$, $B$ and $C$,
  \begin{align}
    \label{eq:Grahams_graphs}
    \iD(\cG_A, \cG_B, \cG_C)
    &= n_-(B-A) - n_-(C-A) + n_-(C-B)
    \\ \nonumber
    &= \left(\cbd \widetilde{n}_-\right)(A, B, C),
  \end{align}
  where $\widetilde{n}_-(X,Y) := n_-(Y-X)$.
\end{theorem}

One application of \eqref{eq:Grahams_graphs} is the following curious
identity: for any invertible $n\times n$ Hermitian
matrices $A$ and $B$,
\begin{equation}
  \label{eq:crazy_formula}
  n_-\big(A-B\big) - n_-\big(B^{-1}-A^{-1}\big) = n_-(A) - n_-(B).
\end{equation}
Non-invertible $A$ or $B$ are discussed in
Section~\ref{sec:applications}.  We also point out that
\eqref{eq:Grahams_graphs} can be used for arbitrary $\cL_1$, $\cL_2$
and $\cL_3$ using symplectic invariance \eqref{eq:sympl_inv} and a
symplectic transformation $\mathsf{S} $ which makes all
$\mathsf{S}\cL_j$ transversal to $\cG_\infty$, for example
\begin{equation}
  \label{eq:Sepsilon}
  \mathsf{S} =
  \begin{pmatrix}
    I & \epsilon I \\
    0 & I
  \end{pmatrix}
\end{equation}
with generic $\epsilon$, see the discussion following \eqref{eq:Edef}.


\begin{example}
  \label{ex:three_graces}
  In the symplectic space $\C \oplus \C$, consider three Lagrangian
  plane, $\cG_A := \{ (z, Az) \colon z\in \C\}$ and the similarly
  defined $\cG_B$ and $\cG_C$, with the ``slopes'' $A,B,C \in \R$.
  Theorem~\ref{thm:general_formula} immediately yields
  \begin{equation}
    \label{eq:basic_iD}
    \iD(\cG_A,\cG_B,\cG_C)
    =
    \begin{cases}
      0,\quad & A \leq B \leq C
      \ \text{ or }\ 
      C < A \leq B
      \ \text{ or }\ 
      B \leq C < A,\\
      1,\quad & A \leq C < B
      \ \text{ or }\
      B < A \leq C
      \ \text{ or }\
      C < B < A.
    \end{cases}
  \end{equation}
  These results (with distinct $A$, $B$ and $C$) are illustrated in
  Figure~\ref{fig:iD_examples}.  The index $\iD$ can be evaluated by
  rotating the line $\cG_A$ counterclockwise to $\cG_B$ and counting
  the number of times (0 or 1 in this low-dimensional case) that it
  intersects $\cG_C$.  The borderline cases are treated according
  to \eqref{eq:basic_iD}: $\cG_A = \cG_C$ counts as an intersection
  but $\cG_B = \cG_C$ does not; if $\cG_A=\cG_B$, the index is 0 for
  any $\cG_C$.  This ``dynamic'' perspective readily generalizes to
  the case when one or more of the Lagrangian planes is vertical, as
  well as to Lagrangian planes in $\C^n\oplus\C^n$ where the
  intersections are counted with their dimension.  \hfill$\Diamond$
\end{example}

\begin{figure}
  \centering
  \includegraphics[scale=0.8]{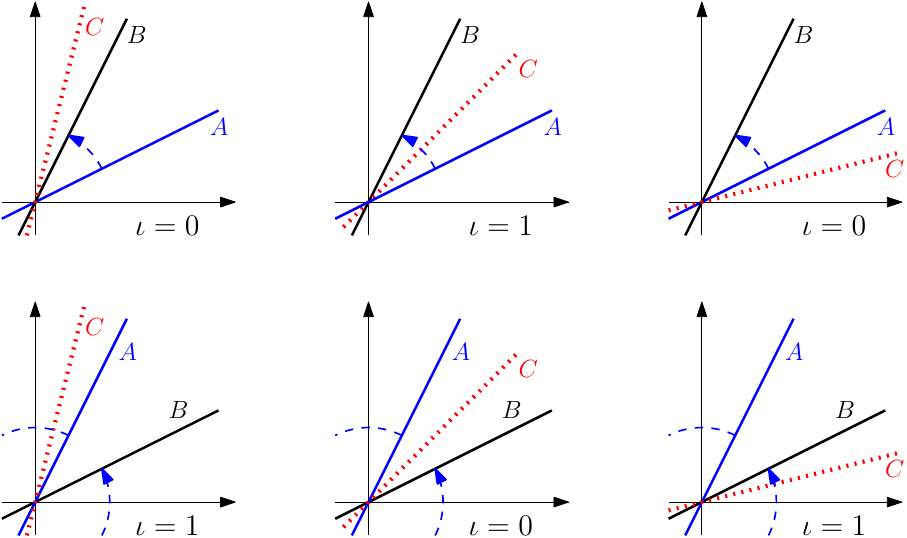}
  \caption{Possible transversal configurations of three Lagrangian
    planes (illustrated in $\R^2$) and the corresponding values of the
    Duistermaat index, see Example~\ref{ex:three_graces}.  The directed
    dashed arcs indicate the non-decreasing paths from $A$ to $B$,
    realizing the minimum in \eqref{eq:MasDuist}.}
  \label{fig:iD_examples}
\end{figure}

The intuition gained in Example~\ref{ex:three_graces} allows us to
interpret the Duistermaat index as the minimal value of the Maslov
index over non-decreasing paths.  This interpretation highlights the
usefulness of the Duistermaat index, since computation of the Maslov
index is far more difficult than applying an algebraic identity such
as \eqref{eq:Grahams_graphs}.
\begin{theorem}
  \label{thm:extremal_Dui}
  For arbitrary Lagrangian planes $\cL_0$, $\cL_1$ and $\cM$ one has
  \begin{equation}
    \label{eq:MasDuist}
    \iD(\cL_0,\cL_1,\cM)
    = \min \left\{\Mas_{[0,1]}\!\big(\cL(t), \cM\big)\Big|\
      \begin{matrix}
	\cL(0)=\cL_0, \, \cL(1)=\cL_1, \\
	t\mapsto\cL(t) \text{\ is  non-decreasing} 
      \end{matrix} \right\},
  \end{equation}
  with the minimum attained by some (non-unique) path $\cL(t)$.
  Furthermore, any path $\cL(t)$ that attains the minimum for one
  $\cM$ will attain the minimum for all $\cM$.
\end{theorem}

Here, $\Mas_{[0,1]}\!\big(\cL(t), \cM\big)$ is the \emph{Maslov index}
as defined in \eqref{eq:MaslovIndex_def}; the term \emph{non-decreasing} is
defined in Section~\ref{SS4.2}.

The plan for the rest of the note is as follows: in
Section~\ref{sec:existence} we show that an index satisfying the
conditions in Theorem~\ref{def:Dui_index} exists and is unique; in
Section~\ref{sec:properties} we establish some further properties of
this index.  In Section~\ref{sec:Maslov} we prove
Theorem~\ref{thm:extremal_Dui}, providing along the way a short proof
(see Theorem~\ref{thm:ZWZ}) of the Zhou--Wu--Zhu identity
\cite[Thm.~1.1]{ZhoWuZhu_fmc18}, which expresses the so-called
\emph{H\"ormander index} as a difference of Duistermaat indices.
Finally, in Section~\ref{sec:applications}, we prove
\eqref{eq:crazy_formula} and its extensions.  We also give an
elementary derivation of \eqref{eq:crazy_formula} using the Haynsworth
inertia formula.

\subsection*{Notation}

We use $x,y$ to denote vectors in $\C^n$ and $u,v$ to denote
vectors in $\C^n\oplus\C^n$. To save space we write $u=(x, y)$ for the
vector 
$u= \left(\begin{smallmatrix} x \\ y \end{smallmatrix} \right) \in\C^n\oplus\C^n$
with components $x,y\in\C^n$. Similarly, we abbreviate
$(X,Y)=\left(\begin{smallmatrix}X\\Y\end{smallmatrix}\right)$
 for the $(2n\times n)$-matrix
with $(n\times n)$-matrix components $X,Y$. We write
$\ker M=\{x\in\C^n: Mx=0\}$ for the kernel of an $(n\times n)$-matrix
$M$ and $\cX \simeq \cY$ when subspaces $\cX$ and $\cY$ are
isomorphic.

\subsection*{Acknowledgements}

This material is based upon work (G.B.) supported by a grant from the
Institute for Advanced Study (IAS), School of Mathematics.  G.B. also
acknowledges support from the NSF Grants DMS-2247473, DMS--2510345 and
from the Association of Former Students of Texas A\&M University
(through the Faculty Development Leave program that supported the
leave to IAS).  G.C. acknowledges the support of NSERC grants
RGPIN-2017-04259 and RGPIN-2025-06186.  Y.L. was supported by the NSF grant
DMS-2106157, and would like to thank the Courant Institute of
Mathematical Sciences at NYU and especially Prof. Lai-Sang Young for
their hospitality. S.S. was supported in part by the NSF Grants
DMS-2418900 and DMS--2510344, Simons Foundation grant
MP--TSM--00002897, grant no. $2024154$ from the U.S.-Israel Binational
Science Foundation Jerusalem, Israel, and by the Office of the Vice
President for Research \& Economic Development (OVPRED) at Auburn
University through the Research Support Program grant.

The authors are extremely grateful to an anonymous referee who read
the paper thoroughly and suggested numerous corrections and
clarifications, substantially improving the manuscript.

\section{Existence and uniqueness of the Duistermaat index}
\label{sec:existence}

In Theorem~\ref{thm:existence_Duistermaat} we define a function $\iD$ that satisfies the conditions of
Theorem~\ref{def:Dui_index}.  Uniqueness then follows from the
observation that these conditions provide a prescription for the
computation of $\iD$; see Section~\ref{ssec:uniqueness} for details.


\subsection{Existence of the index}

\begin{theorem}
  \label{thm:existence_Duistermaat}
  Given a triple $(\cL_1, \cL_2, \cL_3)$ of Lagrangian planes, define
  a Hermitian form $\Qform$ on
  $\cL_1\oplus\cL_2\oplus\cL_3$ by
  \begin{align}
    \label{eq:bigQ}
    \Qform(u_1, u_2, u_3; v_1, v_2, v_3)
    &= \omega(u_1,v_2) +\omega(u_2,v_3) +\omega(u_3,v_1)
    \\ \nonumber
    &\qquad - \omega(u_2,v_1) - \omega(u_3,v_2) - \omega(u_1,v_3).
  \end{align}
  Then the index $\iD(\cL_1, \cL_2, \cL_3)$ defined by
  \begin{equation}
    \label{eq:Dui_via_Q}
    \iD(\cL_1, \cL_2, \cL_3) := n_-(\Qform) - n + \dim \cL_1\cap\cL_3
    \end{equation}
    satisfies the conditions of Theorem~\ref{def:Dui_index}.
\end{theorem}

\begin{remark}
  \label{rem:you_worry_too_much}
  This is more direct than the definition of $\iD$ given in \cite{Dui_am76,ZhoWuZhu_fmc18}, 
  which requires choosing an arbitrary fourth Lagrangian plane.
  One benefit of
  equation \eqref{eq:Dui_via_Q} is that $\iD$ is evidently a
  symplectic invariant. This definition is not optimal
  for practical computations, however, since $\Qform$ is defined on a
  $3n$-dimensional space.  One also notices a lack of symmetry caused
  by the $\dim \cL_1\cap\cL_3$ term.
  Both of
  these concerns are addressed by a more practical (and more
  symmetric) formula \eqref{eq:Dui_index_gen} for the index given in Proposition~\ref{prop:Dui_via_R} below. 
\hfill$\Diamond$\end{remark}

Before proving the theorem, we recall that any $n$-dimensional subspace
$ \cM \subset \C^n \oplus \C^n$ can be described by a \term{frame},
which is an injective linear operator
\begin{equation}
  \label{eq:frame}
  Z = \Lframe{X}{Y} \colon \C^n \to \C^n\oplus\C^n
\end{equation}
whose range is $\cM=\{(Xx,Yx): x\in\C^n\}$. Moreover, $\cM$ is Lagrangian if and only if 
$X^*Y = Y^*X$, as easily seen from \eqref{eq:sympl_form} and \eqref{eq:sympl_complement}.

This description is not unique, but it is easy to see that frames $Z$
and $\tilde Z$ describe the same subspace if and only if
$\tilde Z = ZC$ for some invertible $C \colon \C^n \to
\C^n$. Therefore, the set
\begin{equation}
  \label{eq:Edef}
  E(\cM) := \{ \epsilon \in \R : X + \epsilon Y \text{ is not invertible} \}
\end{equation}
and the so-called \emph{$\epsilon$-Robin map} $R^\epsilon:\C^n \to \C^n$,
\begin{equation}
  \label{eq:almostDTN}
  R^\epsilon := Y(X + \epsilon Y) ^{-1},
  \qquad
  \epsilon \in \R \backslash E(\cM),
\end{equation}
are independent of the choice of frame. When $\cM$
is Lagrangian, the set $E(\cM)$ is finite,
since $\det(X + \epsilon Y)$ is a polynomial in $\epsilon$ that is not
identically zero.\footnote{Setting $\epsilon=i$, we have $\ker(X+iY) =
  \ker(X+iY)^*(X+iY) = \ker(X^*X + Y^*Y) = \ker X \cap \ker Y = \{0\}$, using 
  $X^*Y = Y^*X$ and the injectivity of $Z$, so we conclude that
  $\det(X+iY)\neq0$.}  The Lagrangian condition $X^*Y = Y^*X$ also implies that
\begin{equation}
  \label{eq:HermitianR}
  (X+\epsilon Y)^* R^\epsilon (X+\epsilon Y) = X^*Y + \epsilon Y^* Y
\end{equation}
is Hermitian, therefore $R^\epsilon$ is Hermitian for all $\epsilon \in \R \backslash E(\cM)$.

\newcommand\Xterm[2]{X_{#1}^*Y_{#2}-Y_{#1}^*X_{#2}}
\newcommand\Yterm[2]{Y_{#1}^*X_{#2}-X_{#1}^*Y_{#2}}
\newcommand\Rterm[2]{R_{#1}^\epsilon-R_{#2}^\epsilon}

\begin{proposition}
  \label{prop:Dui_via_R}
  Let $\cL_1$, $\cL_2$ and $\cL_3$ be Lagrangian planes and $R^\epsilon_1$, $R^\epsilon_2$ and $R^\epsilon_3$ the
  corresponding $\epsilon$-Robin maps, for any real
  $\epsilon$ outside the finite set
  $E_{123} := E(\cL_1)\cup E(\cL_2)\cup E(\cL_3)$.
  Then, for the index $\iD$ defined in \eqref{eq:Dui_via_Q}
  and any $\epsilon\in \R\backslash E_{123}$,
  we have
  \begin{align}
    \label{eq:Dui_index_gen}
    \iD(\cL_1,\cL_2,\cL_3) 
    &= n_-\Big(\Rterm21\Big)
    - n_-\Big(\Rterm31\Big)
      + n_-\Big(\Rterm32\Big) \\
    \label{eq:Dui_index_gen_cbd}
    &= \left(\cbd\widetilde{n}_-\right)\left(R_1^\epsilon, R_2^\epsilon,
      R_3^\epsilon\right),
  \end{align}
  where $\widetilde{n}_-(X,Y) := n_-(Y-X)$. In particular, the
  right-hand side of \eqref{eq:Dui_index_gen}, when 
  defined, is independent of $\epsilon$.
\end{proposition}

\begin{proof}
  We follow the proof of a similar identity in \cite[Lem.
  8.3]{Meinrenken_SymplecticGeometry} (attributed therein to Brian
  Feldstein). We first compute the form $\Qform$ defined in \eqref{eq:bigQ} as a
  Hermitian form on $\C^n\oplus\C^n\oplus\C^n$, where each $\C^n$
  parametrizes the corresponding Lagrangian plane $\cL_j$ via a frame
  $(X_j, Y_j)$.  If, for instance,
  \begin{equation}
    \label{eq:xy_uv}
    u = (X_1 x, Y_1 x)
    \in \cL_1, \qquad
    v = (X_2y , Y_2 y)
    \in \cL_2,
  \end{equation}
  then
  \begin{equation}
    \label{eq:wxy}
    \omega(u,v)
    = \big\langle x, (\Xterm12) y \big\rangle_{\C^n}.
  \end{equation}
  Similar calculations for the other terms in \eqref{eq:bigQ} result
  in $\Qform$ being represented by the formula
  \[\Omega(u_1,u_2,u_3;v_1,v_2,v_3)
    =\big\langle
    (x_1,x_2,x_3), W (y_1,y_2,y_3)
    \big\rangle_{\C^n\oplus\C^n\oplus\C^n}, \]
 where $u_j=(X_jx_j,Y_jx_j), v_j=(X_jy_j,Y_jy_j)$, $j\in\{1,2,3\}$, with the Hermitian matrix
  \begin{equation}
    \label{eq:Qmatrix}
    \Qmat =
    \begin{pmatrix}
      0 & \Xterm12 & \Yterm13 \\
      \Yterm21 & 0 & \Xterm23 \\
      \Xterm31 & \Yterm32 & 0
    \end{pmatrix}.
  \end{equation}
  We now note that for $\epsilon \in \R \backslash E_{123}$ and $j,k\in\{1,2,3\}$,
  \begin{align}
    \label{eq:Xterm_R}
    \Xterm{j}{k}
    &= (X_j + \epsilon Y_j)^* Y_k - Y_j^* (X_k + \epsilon Y_k)
    \nonumber \\
    &= (X_j + \epsilon Y_j)^* (\Rterm{k}{j})
      (X_k + \epsilon Y_k),
  \end{align}
  where we used the fact that $R^\epsilon$ is Hermitian.  We now
  verify by direct computation that
  \begin{equation}
    \label{eq:Qmatrix_similar}
    \Qmat = \frac12 D^*T^*
    \begin{pmatrix}
      \Rterm32 & 0 & 0\\
      0 & \Rterm13 & 0 \\
      0 & 0 & \Rterm 21
    \end{pmatrix}
    TD,
  \end{equation}
  where the matrices
  \begin{equation}
    \label{eq:TandD}
    T :=
    \begin{pmatrix}
      -I & \phantom{-}I & \phantom{-}I \\
      \phantom{-}I & -I & \phantom{-}I \\
      \phantom{-}I & \phantom{-}I & -I
    \end{pmatrix},
    \qquad
    D :=
    \begin{pmatrix}
      X_1+\epsilon Y_1 & 0 & 0\\
      0 & X_2+\epsilon Y_2 & 0\\
      0 & 0 & X_3+\epsilon Y_3
    \end{pmatrix}
  \end{equation}
  are invertible for $\epsilon \in \R \backslash E_{123}$.

  We next observe that, by \eqref{eq:Xterm_R} and invertibility
  of $X_k+\epsilon Y_k$,
  \begin{equation}
    \label{eq:kerRX}
    \ker\left(\Rterm{k}{j}\right)
    \simeq \ker\big(\Xterm{j}{k}\big).
  \end{equation}
  We then have
  \begin{align}
    \label{eq:kerR+kerX}
    \ker\big(\Xterm{j}{k}\big)
    &= \big\{y \in \C^n \colon
    0=\left<x, (\Xterm{j}{k}) y \right>_{\C^n}
    \text{ for all } x\in\C^n\big\} \nonumber \\
    &\simeq \big\{v \in \cL_k \colon \omega(u, v) = 0
      \text{ for all } u \in \cL_j \big\} \nonumber \\
    &= \cL_j \cap \cL_k,
  \end{align}
  where we used \eqref{eq:wxy} to get to the second line and the
  definition of a Lagrangian plane in \eqref{eq:sympl_complement} to
  get to the third.  We can now conclude that
  \begin{equation}
    \label{eq:n0R}
    n_0\big(\Rterm{k}{j}\big) = \dim \cL_j \cap \cL_k,
  \end{equation}
  where $n_0(A) = \dim \ker A$.

Finally, applying Sylvester's law of inertia to
  \eqref{eq:Qmatrix_similar}, we use the identity
    \begin{equation}
    \label{eq:symmetry_nminus}
    n_-(\Rterm13) = n_+(\Rterm31) = n - n_0(\Rterm31) - n_-(\Rterm31),
  \end{equation}
along with \eqref{eq:n0R}, to obtain 
 \begin{equation}
    \label{eq:formula_Dui}
    n_-(\Qform) = n - \dim \cL_1 \cap \cL_3 + n_-(\Rterm32) - n_-(\Rterm31) +
      n_-(\Rterm21),
  \end{equation}
  which is equivalent to \eqref{eq:Dui_index_gen}.  
\end{proof}

\begin{remark}
  \label{rem:n0Q}
  For future use we observe that \eqref{eq:Qmatrix_similar} together
  with \eqref{eq:n0R} gives
  \begin{equation}
    \label{eq:n0Q}
    n_0(\Qform) = \dim\cL_1\cap\cL_2 + \dim\cL_1\cap\cL_3 + \dim\cL_2\cap\cL_3.
  \end{equation}
\hfill$\Diamond$\end{remark}

\begin{proof}[Proof of Theorem~\ref{thm:existence_Duistermaat}]
  To check the normalization condition \eqref{eq:normalization}, we apply \eqref{eq:Dui_index_gen}
  with $\cL_1 = \Hor$, $\cL_2 = \cG_A$ and $\cL_3 = \Ver$. For small positive $\epsilon$ we have
  $R_1^\epsilon = 0$, 
  $R_2^\epsilon =A(I+\epsilon A)^{-1} 
  = A + o(1)$ and $R_3^\epsilon = \epsilon^{-1} I$, 
  therefore $R_3^\epsilon-R_2^\epsilon=\epsilon^{-1}(I-\epsilon A+o(\epsilon))$ and so
  \begin{equation}
    \label{eq:n_with_vertical}
    n_-\big( R^\epsilon_2-R^\epsilon_1 \big) = n_-(A),
    \qquad
    n_-\big( R^\epsilon_3-R^\epsilon_2 \big) = 0,
    \qquad
    n_-\big( R^\epsilon_3-R^\epsilon_1 \big) = 0.
  \end{equation}
  The normalization condition then follows from \eqref{eq:Dui_index_gen}.

  Symplectic invariance of $\iD$ follows from the symplectic invariance of every
  term in \eqref{eq:bigQ}, and the cocycle property $\cbd\!\iD=0$ follows from~\eqref{eq:Dui_index_gen_cbd}
  by the well-known identity $\cbd^2 \equiv 0$ for the coboundary
  operator, which is easily verified from \eqref{eq:coboundary_def}.

  Finally, to estimate the range of the possible values of the index,
  we use a symplectic transformation $\mathsf{S}$ such that
  $\mathsf{S}\cL_3 = \cG_\infty$ to reduce the question to
  bounding the range of $\iD(\cL_1,\cL_2,\cG_\infty)$.  That the
  latter is between $0$ and $n$ follows from a version of equation
  \eqref{eq:Dui_index_gen}: 
  \begin{equation}
    \label{eq:Grahams_vertical}
    \iD(\cL_1,\cL_2,\cG_\infty) 
    = n_-\Big(\Rterm21\Big)
    \qquad
    \text{for all }0 < \epsilon \ll 1.
  \end{equation}
  Equation~\eqref{eq:Grahams_vertical} was already established in
  \cite[Eq.~(3.32)]{BerCoxLatSuk_prep23}, but we repeat the proof here
  for completeness.  We use equation~\eqref{eq:Dui_index_gen} with
  $R^\epsilon_3 = \epsilon^{-1} I$, and the desired result follows
  once we establish that
  $n_-\left(\Rterm31\right)=0=n_-\left(\Rterm32\right)$.  In other
  words, we claim that for any Lagrangian frame $(X, Y)^T$, the
  Hermitian operator $\epsilon^{-1} I - Y(X+\epsilon Y)^{-1}$ is
  non-negative definite for small enough $\epsilon>0$.  It is
  equivalent to consider
  \begin{equation}
    \label{eq:similarity}
    \epsilon (X+\epsilon Y)^*
    \left(\epsilon^{-1} I - Y(X+\epsilon Y)^{-1}\right)
    (X+\epsilon Y)
    =  X^* X+\epsilon Y^* X.
  \end{equation}
  The right-hand side is identically zero on $\ker(X^*X) = \ker X$,
  while on $\ker(X^*X)^\perp$ it is a Hermitian perturbation of
  a positive definite matrix.
\end{proof}

\subsection{Proofs of Theorems~\ref{def:Dui_index} and
  \ref{thm:general_formula}}
\label{ssec:uniqueness}

\begin{proof}[Proof of Theorem~\ref{def:Dui_index}]
  Theorem~\ref{thm:existence_Duistermaat} proves the existence of a
  function $\iD$ satisfying the axioms in
  Theorem~\ref{def:Dui_index}. 
%
  Uniqueness follows from the observation that these axioms provide a
  prescription for the computation of $\iD$.  First we reduce the
  computation to the case when $\cL_3$ is transversal to both $\cL_1$
  and $\cL_2$ by using the cocycle identity with \emph{some} $\cL_4$
  that is transversal to the other three.  We can then find a
  symplectic transformation that maps $\cL_3$ to $\Ver$ and $\cL_1$ to
  $\Hor$ (which is well-known to be possible, see for example
  \cite[Exercise 8.1]{Meinrenken_SymplecticGeometry}).  Since this
  transformation preserves transversality, the transformed $\cL_2$ can
  be represented as the graph of a Hermitian matrix, and
  \eqref{eq:normalization} can be used.
\end{proof}

\begin{proof}[Proof of Theorem~\ref{thm:general_formula}]
  Choosing the frame $(I, A)$ for the Lagrangian plane $\cL_1 =
  \cG_A$ and setting $\epsilon=0$, we get $R_1^0 = A$.  The same process
  works for $\cG_B$ and $\cG_C$, and \eqref{eq:Grahams_graphs} follows
  from \eqref{eq:Dui_index_gen}.
\end{proof}

\section{Further properties of the Duistermaat index}
\label{sec:properties}

We now derive useful properties of the Duistermaat index, some of which involve algebraic 
operations that can be defined when Lagrangian subspaces are viewed as 
self-adjoint linear relations.

\subsection{General identities}

\begin{theorem}
  \label{thm:properties}
  The Duistermaat index has the following properties.
  \begin{enumerate}
  \item \label{item:spec_values} (Special values)
    For any Lagrangian $\cL$ and $\cM$,
    \begin{equation}
      \label{eq:spec_values1}
      \iD(\cL, \cL, \cM) = \iD(\cM, \cL, \cL) = 0,  \qquad
      \iD(\cL, \cM, \cL) = n - \dim\cL\cap\cM.
    \end{equation}

  \item \label{item:permutation} (Permuting the arguments)
    \begin{align}
      \label{eq:swap12}
      \iD(\cL_1,\cL_2,\cL_3) + \iD(\cL_2,\cL_1,\cL_3)
      &= n - \dim \cL_1 \cap \cL_2, \\
      \label{eq:swap23}
      \iD(\cL_1,\cL_2,\cL_3) + \iD(\cL_1,\cL_3,\cL_2)
      &= n - \dim \cL_2 \cap \cL_3, \\
      \label{eq:swap13}
      \iD(\cL_1,\cL_2,\cL_3) + \iD(\cL_3,\cL_2,\cL_1)
      &= n - \dim \cL_1 \cap \cL_2 - \dim \cL_2 \cap \cL_3
        + \dim\cL_1\cap\cL_3, \\
      \label{eq:iDshift1}
      \iD(\cL_1,\cL_2,\cL_3) - \dim\cL_1\cap\cL_3
      &= \iD(\cL_3,\cL_1,\cL_2) - \dim\cL_2\cap\cL_3 \\
      \label{eq:iDshift-1}
      &= \iD(\cL_2,\cL_3,\cL_1) - \dim\cL_1\cap\cL_2.
    \end{align}
  \item \label{item:additivity} (Symplectic additivity) For any three
    Lagrangians $\cL_1$, $\cL_2$, $\cL_3$ in a symplectic space $V$
    and any three Lagrangians $\cM_1$, $\cM_2$, $\cM_3$ in a
    symplectic space $W$,
    \begin{equation}
      \label{eq:additivity}
      \iD(\cL_1\oplus\cM_1, \cL_2\oplus\cM_2, \cL_3\oplus\cM_3)
      = \iD(\cL_1,\cL_2,\cL_3) + \iD(\cM_1,\cM_2,\cM_3).
    \end{equation}
  \item \label{item:antisymplectic} (Anti-symplectic transformation)
    If $\mathsf{A}$ is an anti-symplectic linear transformation on
    $(\C^n\oplus\C^n, \omega)$, i.e., a bijection that satisfies
    $\omega(\mathsf{A}u, \mathsf{A}v) = -\omega(u,v)$, then
    \begin{equation}
      \label{eq:antisymplectic}
      \iD(\mathsf{A}\cL_1, \mathsf{A}\cL_2, \mathsf{A}\cL_3)
      = \iD(\cL_3,\cL_2,\cL_1).
    \end{equation}
  \end{enumerate}
\end{theorem}

\begin{proof}
  The first two equalities in \eqref{eq:spec_values1} follow
  directly from \eqref{eq:Dui_index_gen}; for the third we also use
  \eqref{eq:symmetry_nminus} and \eqref{eq:kerR+kerX}.

  To obtain~\eqref{eq:swap12}, we apply the cocycle property
  \begin{align}
    \label{eq:cocycle_for_swap1}
    0 &= (\cbd\!\iD)(\cL_2, \cL_1,\cL_2, \cL_3)
    \\ \nonumber
    &= \iD(\cL_2, \cL_1,\cL_2) - \iD(\cL_2, \cL_1,\cL_3)
    + \iD(\cL_2, \cL_2,\cL_3) - \iD(\cL_1, \cL_2,\cL_3)
  \end{align}
  and use $\iD(\cL_2, \cL_1,\cL_2) = n - \dim\cL_1\cap\cL_2$ and
  $\iD(\cL_2, \cL_2,\cL_3) = 0$.  Equation~\eqref{eq:swap23} follows
  by a similar computation of $(\cbd\!\iD)(\cL_1, \cL_2,\cL_3, \cL_2)$;
  other permutations are obtained by composition of the
  established two.
  
  Symplectic additivity \eqref{eq:additivity} follows by noting that a
  possible choice of frame for $\cL\oplus\cM$ is the direct sum of
  frames.  Therefore, for all but finitely many $\epsilon$, one has
  $R^\epsilon_{\cL\oplus\cM} = R^\epsilon_{\cL} \oplus
  R^\epsilon_{\cM}$.  Equation \eqref{eq:additivity} then follows from
  \eqref{eq:Dui_index_gen} and the additivity of $n_-$ under direct sums.

  Finally, to obtain \eqref{eq:antisymplectic}, we note that the form
  $\Qform$ of \eqref{eq:bigQ} changes sign under an anti-symplectic
  transformation.  Therefore, by \eqref{eq:Dui_via_Q},
  \begin{equation}
    \label{eq:antiQ}
    \iD(\mathsf{A}\cL_1, \mathsf{A}\cL_2, \mathsf{A}\cL_3)
    = n_-(-\Qform) - n + \dim\cL_1\cap\cL_3.
  \end{equation}
  We now use $n_-(-\Qform) = 3n - n_0(\Qform) - n_-(\Qform)$, as well as
  \eqref{eq:n0Q} and \eqref{eq:swap13}, to get the result.
\end{proof}

\subsection{Self-adjoint linear relations}
For a given splitting of the symplectic space (in our case
$\C^n\oplus\C^n$), every Lagrangian plane can be viewed as a
self-adjoint linear relation; see
\cite[Sec.~14.1]{Schmudgen_unboundedSAO} or
\cite[Sec.~4.2]{BooZhu_memAMS}.  In this context, the \emph{difference} of
two Lagrangian planes, $\cL$ and $\cM$, is the Lagrangian plane
\begin{equation}
  \label{eq:difference_def}
  \cL - \cM := \left\{
    (x, y_{_\cL} - y_{_\cM} )  
    \in \C^n \oplus \C^n : (x, y_{_\cL} )  
    \in \cL,\
    (x,  y_{_\cM})  
    \in \cM \right\}.
\end{equation}
This definition is chosen so that
\begin{equation}
  \label{eq:diff_graphs}
  \cG_A - \cG_B = \cG_{A-B}
\end{equation}
for any $n\times n$ Hermitian matrices $A$ and $B$. Another important special case is
\begin{equation}
  \label{eq:diff_inf}
  \cL - \Ver = \Ver = \Ver - \cL
\end{equation}
for any Lagrangian plane $\cL$.

The \emph{inverse} of a Lagrangian plane $\cL$ is
\begin{equation}
	\label{eq:inverse_def}
	\cL^{-1} := \left\{ (y , x) \in \C^n\oplus\C^n : (x,y)
	\in \cL \right\}.
\end{equation}
We define by
\begin{align}
  \label{eq:dom_def}
  \dom\cL &:= \left\{x\in \C^n :
      (x,y)
    \in\cL \text{ for some }y\in\C^n\right\} 
  \\
  \label{eq:mul_def}
  \mul\cL &:= \left\{y \in \C^n :
      (0, y)
    \in\cL \right\}
\end{align}
the \emph{domain} and the \emph{multivalued part} of the linear relation $\cL$. For
instance, if $M$ is an $(n\times n)$ matrix then $\dom\cG_M=\C^n$,
$\mul\cG_M = \{0\}$ and $\mul\cG_M^{-1}=\ker M$; on the other hand $\dom\cG_\infty=\{0\}$ and $\mul\cG_\infty=\C^n$. In general,
we have $\mul\cL = \left(\dom\cL\right)^\perp$ and $\cL$ can be
represented as
\begin{equation}
  \label{eq:gen_Lagr}
  \cL = \big\{(x, Lx+y) \colon x\in \dom\cL,\ y\in \mul\cL\big\}
\end{equation}
for some Hermitian matrix $L \colon \dom \cL \to \dom\cL$, as in
\cite[Thm.~1.5.1]{BehHasDeS_boundarytriples} or
\cite[Prop.~14.2]{Schmudgen_unboundedSAO}.

\begin{corollary}
  \label{cor:SALR}
  The Duistermaat index enjoys the following further properties.
  \begin{enumerate}
  \item \label{item:subtraction} For any Hermitian
    matrix $A$,
    \begin{equation}
      \label{eq:subtraction}
      \iD(\cL_1,\cL_2,\cL_3)
      = \iD(\cL_1-\cG_A,\cL_2-\cG_A,\cL_3-\cG_A).
    \end{equation}
  \item \label{item:inversion} For any Lagrangian planes $\cL_1$, $\cL_2$ and $\cL_3$,
    \begin{equation}
      \label{eq:inversion}
      \iD(\cL_1,\cL_2,\cL_3) = \iD(\cL_3^{-1},\cL_2^{-1},\cL_1^{-1}).
    \end{equation}
  \item \label{item:berndt-luger} For any Hermitian $A$ and Lagrangian
    $\cL$, represented as in \eqref{eq:gen_Lagr},
  \begin{align}
    \label{eq:berndt-luger}
    &\iD(\cG_A, \cL, \Ver) = n_-\big(L - A_{\dom\cL}\big), \\
    \label{eq:luger-berndt}
    &\iD(\cL, \cG_A, \Ver) = n_-\big(A_{\dom\cL} - L\big)
      + \dim \cL \cap \cG_\infty,
  \end{align}
  where $A_{\dom\cL}$ denotes the compression of $A$ to $\dom\cL$, namely
  \begin{equation}
    \label{eq:compression_def}
    A_{\dom\cL} := P A|_{\dom\cL},
  \end{equation}
  where $P$ is the orthogonal projector in $\C^n$ onto $\dom\cL$.
  \end{enumerate}
\end{corollary}

\begin{proof}
  Part (\ref{item:subtraction}) follows from symplectic invariance with
  \begin{equation}
    \label{eq:Sagain}
    \mathsf{S} : =
    \begin{pmatrix}
      I & 0\\
      A & I
    \end{pmatrix},
  \end{equation}
  which maps $\cL - \cG_A \mapsto \cL$.

  For Part (\ref{item:inversion}) we note that the mapping $(x,y) \mapsto (y,x)$ on
  $\C^n\oplus\C^n$, used in the definition of $\cL^{-1}$, is
  anti-symplectic.

  \newcommand{\cGLp}{{\cG_L'}}
  \newcommand{\Verp}{{\cG_\infty'}}
  \newcommand{\Verpp}{{\cG_\infty''}}
  \newcommand{\Horp}{{\cG_0'}}
  \newcommand{\Horpp}{{\cG_0''}}
  
  For Part (\ref{item:berndt-luger}) we use \eqref{eq:subtraction} to
  get
  \begin{equation}
    \label{eq:blb}
    \iD(\cG_A, \cL, \Ver) = \iD(\cG_0, \cL-\cG_A, \Ver),
    \qquad
    \iD(\cL, \cG_A, \Ver) = \iD(\cL-\cG_A, \cG_0, \Ver).
  \end{equation}
  We observe that
  \begin{align}
    \label{eq:LminusA}
    \cL - \cG_A
    &= \big\{(x, Lx-Ax+y) \colon x\in\dom\cL,\ y\in\mul\cL \big\}
    \\ \nonumber
    &= \big\{(x, (L-PA)x+y') \colon x\in\dom\cL,\ y'\in\mul\cL \big\}
    \\ \nonumber
    &= \big\{(x, (L-A_{\dom\cL})x+y') \colon x\in\dom\cL,\ y'\in\mul\cL \big\}.
  \end{align}
  This shows that it is enough to establish the $A=0$ versions of 
  \eqref{eq:berndt-luger} and \eqref{eq:luger-berndt}, namely
  \begin{equation}
    \label{eq:BL0}
    \iD(\Hor, \cL, \Ver) = n_-(L), \qquad
    \iD(\cL, \Hor, \Ver) = n_-(- L) + \dim \cL \cap \cG_\infty.
  \end{equation}
  We now partition the symplectic space as
  \begin{equation}
    \label{eq:space_partition}
    \C^n\oplus\C^n = \big(\dom\cL\oplus \dom\cL\big)
    \oplus \big(\mul\cL\oplus \mul\cL\big) =: V' \oplus V''.
  \end{equation}
  It is straightforward to check that
  the spaces $V'$ and $V''$ are symplectic\footnote{The restriction of $\omega$ to $V'$ is the 
  canonical symplectic form induced by the inner product on $\dom\cL$ (cf. \eqref{eq:sympl_form}), and likewise for the restriction of $\omega$ to $V''$.} with the forms
  $\omega' := \omega|_{V'}$ and $\omega'' := \omega|_{V''}$, and that
  $\omega = \omega'\oplus\omega''$.  With respect to this partition,
  we have $\Hor = \Horp \oplus \Horpp$, $\cL = \cGLp \oplus \Verpp$
  and $\Ver = \Verp\oplus\Verpp$.  We then use symplectic additivity
  \eqref{eq:additivity}, as well as the special values
  \eqref{eq:spec_values1} and normalization \eqref{eq:normalization},
  to compute
  \begin{align}
    \label{eq:split_off_infinity}
    \iD(\Hor, \cL, \Ver)
    &= \iD_{V'}(\Horp, \cGLp, \Verp)
      + \iD_{V''}(\Horpp, \Verpp, \Verpp)
    \\ \nonumber
    &= \iD_{V'}(\Horp, \cGLp, \Verp)  = n_-(L).
  \end{align}
  Similarly, using the modified normalization $\iD(\cG_A, \Hor, \Ver)
  = \iD(\Hor, \cG_{-A}, \Ver) = n_-(-A)$, obtained from
  \eqref{eq:normalization} and  \eqref{eq:subtraction}, we obtain
  \begin{align}
    \label{eq:split_off_again}
    \iD(\cL, \Hor, \Ver)
    &= \iD_{V'}(\cGLp, \Horp, \Verp)
      + \iD_{V''}(\Verpp, \Horpp, \Verpp)
    \\ \nonumber
    &= \iD_{V'}(\cGLp, \Horp, \Verp) + \frac12\dim V''
      = n_-(-L) + \dim \cL \cap \Ver,
  \end{align}
  which completes the proof.
\end{proof}

\section{Relationship with other indices}
\label{sec:Maslov}

\subsection{The Kashiwara index}

Another commonly used index of a triple of Lagrangian planes arose
independently in the work of H\"ormander, Kashiwara and Wall; see, for
instance, \cite{Wal_im69}, \cite[Eq.~(3.3.15)]{Hor_am71},
\cite[Def.~1.5.1]{LionVergne_WeilMaslov},
\cite[Sec.~8]{CapLeeMil_cpam94}, \cite[Sec 4.2]{Fur_jgp04}.  The
H\"ormander--Kashiwara--Wall index $s(\cL_1, \cL_2, \cL_3)$ is
defined\footnote{Only the reference \cite{Fur_jgp04}
  considered complex spaces; the signs in our definition of the form
  $\Qform$ are different.} as the signature of the form $\Qform$ from
Theorem~\ref{thm:existence_Duistermaat},
\begin{equation}
  \label{eq:Kashiwara_def}
  s(\cL_1, \cL_2, \cL_3) := \Sig(\Qform) = n_+(\Qform) - n_-(\Qform).
\end{equation}
The advantage of dealing with $\Sig(A)$ is its symmetry property
$\Sig(-A) = -\Sig(A)$, in contrast to \eqref{eq:symmetry_nminus}.
However, our motivation comes,
ultimately, from operators that are bounded only from below; in this
case $n_-$ may be well-defined while $n_+$ will not be.

The two indices are connected through the following formulas.

\begin{lemma}
  \label{lem:Kashi_Dui}
  For any Lagrangian planes $\cL_1$, $\cL_2$ and $\cL_3$,
  \begin{equation}
    \label{eq:Kashiwara_from_Dui}
    s(\cL_1, \cL_2, \cL_3) = \iD(\cL_2, \cL_1, \cL_3) - \iD(\cL_3, \cL_1, \cL_2)
  \end{equation}
  or, equivalently,
  \begin{equation}
    \label{eq:Dui_via_Kashiwara}
    \iD(\cL_1, \cL_2, \cL_3) 
    = \frac12 \Big(n - \dim\cL_1\cap\cL_2 + \dim\cL_1\cap\cL_3 - \dim\cL_2\cap\cL_3
      - s(\cL_1, \cL_2, \cL_3) \Big).
  \end{equation}
\end{lemma}

\begin{proof}
  Applying \eqref{eq:Dui_index_gen} to $\iD(\cL_2, \cL_1, \cL_3)$ and
  $\iD(\cL_3, \cL_1, \cL_2)$ and recalling that $n_-(A) = n_+(-A)$, we
  get
  \begin{equation}
    \label{eq:sigReps}
    \iD(\cL_2, \cL_1, \cL_3) - \iD(\cL_3, \cL_1, \cL_2)
    = \Sig(\Rterm21) + \Sig(\Rterm31) + \Sig(\Rterm32),
  \end{equation}
  which is equal to $s(\cL_1, \cL_2, \cL_3)$ by
  \eqref{eq:Kashiwara_def} and \eqref{eq:Qmatrix_similar}.  To obtain
  \eqref{eq:Dui_via_Kashiwara} we use
  Theorem~\ref{thm:properties}(\ref{item:permutation}).
\end{proof}

It follows that $\iD(\cL_1, \cL_2, \cL_3)$ is the only non-trivial
invariant of a triple of Lagrangian planes, in the following sense.

\begin{proposition}
  \label{prop:invariant}
  For any two triples $(\cL_1, \cL_2, \cL_3)$ and
  $(\cM_1, \cM_2, \cM_3)$ of Lagrangian planes, there is a linear
  symplectic transformation $\mathsf{S}$ satisfying
  $\mathsf{S} \cL_j = \cM_j$ for all $j\in\{1,2,3\}$ if and only if
  \begin{align}
    \label{eq:all_invariants3}
    &\dim \cL_j \cap \cL_k = \dim \cM_j \cap \cM_k
      \quad\text{for all }j,k\in\{1,2,3\},\\
    &\dim \cL_1 \cap \cL_2 \cap \cL_3 = \dim \cM_1 \cap \cM_2 \cap
      \cM_3,\\
    \label{eq:nontriv_invariant}
    &\iD(\cL_1, \cL_2, \cL_3) = \iD(\cM_1, \cM_2,
      \cM_3).
  \end{align}
\end{proposition}

\begin{proof}
  Proposition~\ref{prop:invariant} is well known when
  \eqref{eq:nontriv_invariant} is replaced by the condition
  $s(\cL_1, \cL_2, \cL_3) = s(\cM_1, \cM_2, \cM_3)$, see
  \cite[Thm.~8.4(d)]{Meinrenken_SymplecticGeometry} or
  \cite[Prop.~4.4]{AgrGam_SymplMeth}.  By Lemma~\ref{lem:Kashi_Dui} the indices $s$ and $\iD$
  are related by an invertible transformation involving the
  invariants in \eqref{eq:all_invariants3} and so, assuming \eqref{eq:all_invariants3} holds, $s(\cL_1, \cL_2, \cL_3) = s(\cM_1, \cM_2, \cM_3)$ if and only if $\iD(\cL_1, \cL_2, \cL_3) = \iD(\cM_1, \cM_2, \cM_3)$.
\end{proof}

\subsection{The Maslov index}\label{SS4.2}

Suppose $\cL \colon [0,1] \to \Lambda(n)$ is a differentiable
path of Lagrangian planes, given
by a frame $\big(X(t), Y(t)\big)$. For
$t_0 \in [0,1]$ and $v \in \cL(t_0)$ we define 
the\footnote{The equivalence of \eqref{eq:crossing_via_frame} to other common formulas in the literature is proved in \cite[Thm.~2.1]{BerCoxLatSuk_prep23}.}
 \emph{crossing form} $\form$ on $\cL(t_0)$ by 
\begin{equation}
  \label{eq:crossing_via_frame}
  \form(v,v) = \left<x,
    \big(X^*(t_0) Y'(t_0) - Y^*(t_0) X'(t_0)\big) x
  \right>_{\C^n},
  \quad\text{where }
  v =
    \big(X(t_0)x, 
    Y(t_0)x\big).
\end{equation}
We say that the path
$\cL(\cdot)$ is \emph{non-decreasing} if $\form$ is positive
semidefinite at all $t$. We say that it is \emph{regular}, with respect
to a given Lagrangian plane $\cM$, if it is continuously differentable
and the restricted quadratic form $\form|_{\cL(t_0)\cap\cM}$ is
non-degenerate whenever $\cL(t_0)\cap\cM$ is nontrivial. Such $t_0$
are referred to as \emph{crossings} or \emph{conjugate points}.

For a regular path $\cL(\cdot)$ the conjugate points are isolated, and its Maslov index
with respect to a Lagrangian reference plane $\cM$ is given by
\begin{equation}
  \label{eq:MaslovIndex_def}
  \Mas_{[0,1]}\!\big(\cL(t), \cM\big)
  := \sum_{t\in[0,1)}n_+(\form|_{\cL(t)\cap\cM})
  - \sum_{t\in(0,1]}n_-(\form|_{\cL(t)\cap\cM}).
\end{equation} 

\begin{remark}
There are two ways of viewing the formula \eqref{eq:MaslovIndex_def}, depending on 
how one defines the Maslov index. One approach is to define the Maslov index for 
any continuous path via the spectral flow of a certain family of unitary operators. For a regular 
path one then recovers \eqref{eq:MaslovIndex_def} as a special case; see \cite{BooFur_tjm98,Fur_jgp04} for details. 
Alternately, one can take \eqref{eq:MaslovIndex_def} as the \emph{definition} of the Maslov index for a regular path, 
then define the Maslov index of a continuous path by first deforming it to a regular path and then arguing that the 
resulting index is independent of the choice of deformation, as in \cite{RobSal_t93}.
\hfill$\Diamond$\end{remark}

\begin{lemma}
  \label{lem:minimal_path}
  For any Lagrangian planes $\cL_0$ and $\cL_1$, there exists a
  smooth, non-decreasing path $\cL \colon [0,1] \to \Lambda(n)$ between $\cL_0$ and $\cL_1$
  such that
  \begin{equation}
    \label{eq:minimal_path}
    \iD(\cL_0,\cL_1,\cM) = \Mas_{[0,1]}\!\big(\cL(t), \cM\big)
  \end{equation}
  for any Lagrangian plane $\cM$.
\end{lemma}

\begin{proof}
  Let $Q \colon \C^n \to \C^n$ be an arbitrary orthogonal projector
  with $\dim\ker Q = \dim \cL_0 \cap \cL_1$.  Since $\cG_0 \cap \cG_Q$
  and $\cL_0 \cap \cL_1$ have the same dimension, there exists a
  symplectic transformation $\mathsf{S}$ so that
  $\mathsf{S}\cL_0 = \cG_0$ and $\mathsf{S}\cL_1 = \cG_Q$ (see, for
  example, \cite[Exercise 8.1]{Meinrenken_SymplecticGeometry}).  Since
  both quantities in \eqref{eq:minimal_path} are symplectic
  invariants, without loss of generality we can assume that
  $\mathsf{S}$ is the identity. That is, it suffices to find a
  non-decreasing path $\cL(\cdot)$ with $\cL(0) = \cG_0$,
  $\cL(1) = \cG_Q$ and
  $\iD(\cG_0,\cG_Q,\cM) = \Mas\!\big(\cL(t), \cM\big)$ for any
  Lagrangian plane $\cM$. We claim that $\cL(t) = \cG_{tQ}$ has this
  property.

  Representing $\cM$ using a Hermitian
  matrix $M$ on $\dom\cM$, as in \eqref{eq:gen_Lagr}, and letting $Q_{\dom\cM}$ denote the compression of $Q$ to  $\dom\cM$ (see \eqref{eq:compression_def}), we use the cocycle property
  \eqref{eq:cocyle_identity} to obtain
  \begin{align}
    \label{eq:index_with0}
    \iD(\cG_0, \cG_Q, \cM)
    &= \iD(\cG_0, \cG_Q, \cG_\infty) - \iD(\cG_0, \cM, \cG_\infty)
    + \iD(\cG_Q, \cM, \cG_\infty) \\
    &= n_-(Q) 
    - n_-(M) + n_-\big(M-Q_{\dom\cM}\big),
  \end{align}
  where we used \eqref{eq:normalization} for the first term on the
  right-hand side and \eqref{eq:berndt-luger} for the remaining
  terms. Since $n_-(Q) = 0$, the proof of \eqref{eq:minimal_path} will
  be complete once we show that
    \begin{equation}
    \label{eq:Maslov_result}
    \Mas_{[0,1]}\!\big(\cL(t), \cM\big) = n_-\big(M-Q_{\dom\cM}\big) - n_-(M).
  \end{equation}

  To calculate this Maslov index we define the two-parameter family
  $\cL(s,t) = \cG_{tQ - s I}$. It is clear that $\cL(t) = \cL(0,t)$ is
  homotopic, with fixed endpoints, to the concatenation of
  $\cL(s,0)|_{s\in[0,s_*]}$, $\cL(s_*,t)|_{t\in[0,1]}$ and
  $\cL(s_* - s,1)|_{s\in[0,s_*]}$, so the index of $\cL(t)$ is the sum of
  these three indices.  We choose the parameter $s_* > 0$ to be small
  enough to guarantee that neither $M$ nor $M-Q_{\dom\cM}$ has any
  spectrum in the interval $[-s_*,0)$.

  Recalling the representation of $\cM$ in \eqref{eq:gen_Lagr} and
  computing as in \eqref{eq:LminusA}, we
  observe that
  \begin{equation}
    \label{eq:intersection_Mpath2}
    \cL(s,t) \cap \cM = \big\{ (x, Mx) : x \in \dom\cM
    \text{ and } Mx = t Qx - sx \big\}.
  \end{equation}
  Our condition on $s_*$ means that this intersection is trivial for
  $t=0$ and $s \in (0,s_*]$, as well as for $t=1$ and $s \in (0,s_*]$.

  The path $\cL(s,0)|_{[0,s_*]}$ may still have a conjugate point at
  $s = 0$.  However, the $s$-crossing form, computed
  via~\eqref{eq:crossing_via_frame}, is $-\left<x, x\right>_{\C^n}$,
  which is negative definite and hence generates no contribution to
  \eqref{eq:MaslovIndex_def} at the initial point of a path.  We
  conclude that $\Mas_{[0,s_*]}\!\big(\cL(s,0), \cM\big) = 0$ and, by
  similar reasoning,
  $\Mas_{[0,s_*]}\!\big(\cL(s_* - s,1), \cM\big) = 0$.

  We now evaluate the remaining term,
  $\Mas_{[0,1]}\!\big(\cL(s_*,t), \cM\big)$.  The $t$-crossing form
  evaluates to $\left<x, Qx\right>_{\C^n}$, which we claim to be
  strictly positive on $\cL(s_*,t) \cap \cM$ whenever this intersection is
  non-trivial.  Indeed, the form is obviously non-negative and if
  $Qx = 0$ for some nonzero $x \in \dom\cM$, then
  \eqref{eq:intersection_Mpath2} implies $Mx=-s_*x$, which is excluded
  by our condition that $M$ has no spectrum in $[-s_*,0)$, thus proving the claim. The same argument shows that any eigenvalue curve $\lambda(t)$ for the matrix $M - t Q_{\dom\cM} + s_*I_{\dom\cM}$ has $\lambda'(t) \leq 0$ in general and $\lambda'(t_0) < 0$ at any time $t_0$ for which $\lambda(t_0) = 0$.

  It follows that each $n_+$ term in \eqref{eq:MaslovIndex_def} equals
  the dimension of the space $\cL(s_*,t) \cap \cM$, so
  \begin{equation}
    \label{eq:Maslov_dimsum}
    \Mas_{[0,1]}\!\big(\cL(s_*,t), \cM\big)
    = \sum_{t \in [0,1)} \dim\ker \big(M - t Q_{\dom\cM} + s_*I_{\dom\cM}\big).
  \end{equation}
  Since the eigenvalues of $M - t Q_{\dom\cM} + s_*I_{\dom\cM}$ are strictly decreasing in $t$ when they equal zero,
  the sum in \eqref{eq:Maslov_dimsum} is equal to the number of
  eigenvalues of $M - t Q_{\dom\cM} + s_* I_{\dom\cM}$ that pass through zero (and hence become negative) at some $t \in [0,1)$,
  thus
  \begin{equation}
    \Mas_{[0,1]}\!\big(\cL(s_*,t), \cM\big)
    = n_-\big(M-Q_{\dom\cM} + s_* I\big) - n_-(M + s_* I_{\dom\cM}).
  \end{equation}
  Finally, we observe that our conditions on the smallness of $s_*$ imply
  \begin{equation}
    \label{eq:s*_def}
    n_-\big(M-Q_{\dom\cM} + s_* I_{\dom\cM}\big) = n_-\big(M-Q_{\dom\cM}\big),
    \qquad
    n_-(M + s_* I) = n_-(M),
  \end{equation}
  establishing \eqref{eq:Maslov_result}.
\end{proof}

\begin{remark}
\label{rem:sdef}
We note that $\ker M \cap \ker Q\not=\{0\}$ implies $\cL(0,t) \cap \cM\not=\{0\}$ for all $t$, because $(\ker M \cap \ker Q)\oplus\{0\}\subset \cL(0,t) \cap \cM$. This means the crossings are not isolated and the crossing form is degenerate, so \eqref{eq:MaslovIndex_def} does not apply. This is the reason for the introduction of $s$ in the proof above. A similar homotopy argument was used in \cite{RobSal_t93} to prove the ``zero axiom" of the Maslov index, which says that a path $\cL(t)$ for which $\dim(\cL(t) \cap \cM)$ is constant has $\Mas\!\big(\cL(t), \cM\big) = 0$.
\hfill$\Diamond$\end{remark}

As a corollary of Lemma~\ref{lem:minimal_path} we obtain the
Zhou--Wu--Zhu identity \cite[Thm.~1.1]{ZhoWuZhu_fmc18}, which has
numerous applications in the analysis of differential equations
\cite{BarOffPorWu_mz21,How_jmaa21,ElySepSim_f23,BerCoxLatSuk_prep23,LiuZhaZha_dcds24}.

\begin{theorem}[Zhou--Wu--Zhu identity]
  \label{thm:ZWZ}
  For any path $\cL \colon [0,1] \to \Lambda(n)$ and any Lagrangian
  planes $\cM_1$ and $\cM_2$,
  \begin{align}
    \label{eq:Hor_ZWZ0}
    \Mas_{[0,1]}\!\big(\cL(t), \cM_1\big) -
    \Mas_{[0,1]}\!\big(\cL(t), \cM_2\big)
    &= \iD\!\big(\cL(0), \cL(1), \cM_1 \big)
      - \iD\!\big(\cL(0), \cL(1), \cM_2 \big)
    \\ \label{eq:Hor_ZWZ1}
    &= \iD\!\big(\cL(0), \cM_1, \cM_2 \big)
      - \iD\!\big(\cL(1), \cM_1, \cM_2 \big).
  \end{align} 
\end{theorem}

\begin{proof}
  Use Lemma~\ref{lem:minimal_path} to choose a path
  $\widetilde{\cL} \colon [1,2] \to \Lambda(n)$ from $\cL(1)$ to $\cL(0)$
  such that
  \begin{equation}
    \label{eq:minimal_pathM}
    \Mas_{[1,2]}\!\big(\widetilde{\cL}(t), \cM_j\big)
    = \iD\!\big(\cL(1), \cL(0), \cM_j\big),
    \qquad
    j=1,2.
  \end{equation}
  The concatenation of $\cL$ and $\widetilde{\cL}$ is a loop, therefore its Maslov index
  does not depend on the reference plane $\cM_j$.  Using the
  additivity of the Maslov index under concatenation, we get
  \begin{equation}
    \label{eq:concatMaslov}
    \Mas_{[0,1]}\!\big(\cL(t), \cM_1\big)
    + \iD\!\big(\cL(1), \cL(0), \cM_1\big)
    = \Mas_{[0,1]}\!\big(\cL(t), \cM_2\big)
    + \iD\!\big(\cL(1), \cL(0), \cM_2\big).
  \end{equation}
  Equation~\eqref{eq:Hor_ZWZ0} now follows from
  \eqref{eq:swap12}, and equation \eqref{eq:Hor_ZWZ1} follows from the
  cocycle property \eqref{eq:cocyle_identity}.
\end{proof}

We are now ready to proved the relationship \eqref{eq:MasDuist} between the Duistermaat and Maslov indices that was claimed in the introduction.

\begin{proof}[Proof of Theorem~\ref{thm:extremal_Dui}]
  For any non-decreasing path $\cL(t)$ from $\cL_0$ to $\cL_1$, we use
  \eqref{eq:Hor_ZWZ0} with $\cM_1=\cM$, $\cM_2 = \cL(1)$ and note that
  $\iD\!\big(\cL(0),\cL(1),\cL(1)\big)=0$ by \eqref{eq:spec_values1}.  We obtain
  \begin{equation}
    \label{eq:nondec_path_inequality}
    \iD\!\big(\cL(0),\cL(1),\cM\big)
    = \Mas_{[0,1]}\!\big(\cL(t), \cM\big) - \Mas_{[0,1]}\!\big(\cL(t), \cL(1)\big)
    \leq \Mas_{[0,1]}\!\big(\cL(t), \cM\big),
  \end{equation}
  because the path $\cL(\cdot)$ is non-decreasing.
  Lemma~\ref{lem:minimal_path} shows that the upper bound is
  achievable.  Finally, once the path $\cL(t)$ achieves equality for
  some $\cM$, i.e.,
  \begin{equation}
    \label{eq:nondec_equality}
    \iD\!\big(\cL(0),\cL(1),\cM\big) = \Mas_{[0,1]}\!\big(\cL(t), \cM\big),
  \end{equation}
  the same equality holds for any other $\cM$ by \eqref{eq:Hor_ZWZ0}.
\end{proof}

\section{The Morse index of a difference of Hermitian matrices}
\label{sec:applications}

We are now ready to prove equation \eqref{eq:crazy_formula} and its variants. For a Hermitian matrix $X$, we let $X^+$ denote the Moore--Penrose pseudoinverse and let $P_X$ denote the orthogonal projector
  onto $\Ran X$.

\begin{proposition}
  \label{prop:kerAkerB}
   Let $A$ and $B$ be $n \times n$ Hermitian matrices.
  \begin{enumerate}
  \item If 
    $\ker A \subseteq \ker B$, then
    \begin{equation}
      \label{eq:crazy_formula_kerB_bigger}
      n_-\big(A-B\big) = n_-(A) - n_-(B)
      + n_-\big(B^+ - P_B A^+ P_B\big).
    \end{equation}
  \item If $\ker B \subseteq \ker A$, then
    \begin{equation}
      \label{eq:crazy_formula_kerA_bigger}
      n_-\big(A-B\big) = n_-(A) - n_-(B)
      + n_-\big(P_A B^+ P_A - A^+\big) + n_0(A) - n_0(B).
    \end{equation}
  \end{enumerate}
\end{proposition}

We note that equation~\eqref{eq:crazy_formula} follows from either one
of \eqref{eq:crazy_formula_kerB_bigger} and
\eqref{eq:crazy_formula_kerA_bigger}: for an invertible $X$, $P_X$ is the
identity and $X^+=X^{-1}$.

\begin{proof}
  In the case when $\ker A \subseteq \ker B$, the Hermitian matrices
  $A$, $B$, $A-B$ and $B^+ - P_B A^+ P_B$ are all identically zero on
  $\ker A$.  This means we can restrict $A$ and $B$ (and thus $A-B$ and
  $B^+ - P_B A^+ P_B$) to the subspace $(\ker A)^\perp$, on which $A$
  is invertible, without changing their Morse indices. It thus suffices to prove \eqref{eq:crazy_formula_kerB_bigger}
when $A$ is invertible.

  Using
  Theorem~\ref{thm:general_formula}, we just need to show that
  $\iD(\cG_0, \cG_B, \cG_A) = n_-\big(B^+ - P_B A^+ P_B\big)$.  We use
  \eqref{eq:inversion} to get
  \begin{equation}
    \label{eq:inversion_applied}
    \iD(\cG_0, \cG_B, \cG_A) = \iD(\cG_A^{-1}, \cG_B^{-1}, \Ver),
  \end{equation}
  and observe that for invertible $A$ we have
  \begin{equation}
    \label{eq:ABinverse}
    \cG_A^{-1} = \cG_{A^{-1}} = \cG_{A^+}, \qquad
    \cG_B^{-1} = \big\{(x, B^+x+y) \colon x \in \Ran B,\ y \in \ker B \big\},
  \end{equation}
  cf.~\eqref{eq:gen_Lagr}. The desired conclusion now follows from equation~\eqref{eq:berndt-luger}.

The proof of \eqref{eq:crazy_formula_kerA_bigger} proceeds similarly, by restricting all matrices in the equation to $(\ker B)^\perp$  and observing that
  \begin{equation}
    \label{eq:n0unchanged}
    n_0(A) - n_0(B) = n_0\left(A|_{(\ker B)^\perp}\right) -
    n_0\left(B|_{(\ker B)^\perp}\right)
    = n_0\left(A|_{(\ker B)^\perp}\right).
  \end{equation}
  To establish \eqref{eq:crazy_formula_kerA_bigger} under the
  assumption that $B$ is invertible, we use
  Theorem~\ref{thm:general_formula}, equation
  \eqref{eq:inversion_applied}, an analogue of \eqref{eq:ABinverse}
  with $A$ and $B$ swapped and, finally, \eqref{eq:luger-berndt}.
\end{proof}

\begin{corollary}
  \label{cor:crazy_formula_plus}
  For any invertible $n\times n$ Hermitian matrices $A$ and $B$,
  \begin{equation}
    \label{eq:crazy_formula_plus}
    n_-\big(A+B\big) + n_0\big(A+B\big)+ n_-\big(A^{-1}+B^{-1}\big)
    = n_-(A) + n_-(B).
  \end{equation}
\end{corollary}

\begin{proof}
  Equation~\eqref{eq:crazy_formula} with $-A$ in place of $A$ becomes
  \begin{equation}
    \label{eq:crazy_mod}
    n_+\big(A+B\big) - n_-\big(A^{-1}+B^{-1}\big)
    = n_+(A) - n_-(B).
  \end{equation}
  We then use $n_+(X) = n-n_-(X)-n_0(X)$ to obtain
  \eqref{eq:crazy_formula_plus}; $n_0(A)$ disappears because $A$ is invertible.
\end{proof}

Similar computations lead to a much simpler proof of
\cite[Prop.~6.1]{BerCoxLatSuk_prep23}, which we give below.

\begin{proposition}[Prop.~6.1 of \cite{BerCoxLatSuk_prep23}]
  Assume $\cL_1$, $\cL_2$ and $\cL_3$ are Lagrangian planes such that
  $\cL_3$ is transversal to $\cL_1$, $\cL_2$ and $\Ver$.  Then
  \begin{equation}
    \label{eq:index_Delta}
    \iD(\cL_1,\cL_2,\cL_3)
    = n_-\left( (\cL_1 - \cL_3)^{-1} - (\cL_2 - \cL_3)^{-1}\right);
  \end{equation}
  the relation in the right-hand side 
  is the graph of a matrix and $n_-$ denotes its Morse index.
\end{proposition}

\begin{proof}
  We have the following chain of equalities:
  \begin{align}
    \iD(\cL_1,\cL_2,\cL_3)
    &= \iD(\cL_3,\cL_1,\cL_2) \\
      &= \iD(\Hor,\cL_1-\cL_3,\cL_2-\cL_3) \\
    &= \iD\left( (\cL_2-\cL_3)^{-1}, (\cL_1-\cL_3)^{-1}, \Ver\right) \\
    &=  n_-\left( (\cL_1 - \cL_3)^{-1} - (\cL_2 - \cL_3)^{-1}\right).
  \end{align}
  In the first step we applied \eqref{eq:iDshift1} and used
  $\cL_1\cap\cL_3 = \cL_2\cap\cL_3 = \{0\}$.  In the second step we
  used \eqref{eq:subtraction}, which is allowed because $\cL_3$ is transversal
  to $\Ver$ and therefore is the graph of a matrix.  In the third step
  we used \eqref{eq:inversion}.  We remark that because $\cL_3$ is
  transversal to $\cL_1$, $\cL_1-\cL_3$ is transversal to $\Hor$ and
  thus $(\cL_1 - \cL_3)^{-1}$ is transversal to $\Ver$.  Similarly,
  $(\cL_2 - \cL_3)^{-1}$ is also a graph.  Therefore, in the last
  step, we used \eqref{eq:subtraction} to subtract $(\cL_2 -
  \cL_3)^{-1}$ and used the normalization condition \eqref{eq:normalization}.
\end{proof}

\begin{remark}
The identity \eqref{eq:crazy_formula} can also be obtained using Haynsworth's 
inertia additivity formula.
For a block Hermitian matrix $H = \left(\begin{smallmatrix}
    A & B \\ B^* & D
  \end{smallmatrix}\right)$
with invertible $A$, the Haynsworth formula \cite{Hay_laa68} states
\begin{equation}
  \label{eq:Haynsworth}
  n_-(H) = n_-(A) + n_-\left(D - B^* A^{-1} B\right).
\end{equation}
A similar identity holds if $D$ is invertible. Applying this formula in two possible ways to the matrix
$\left(\begin{smallmatrix}
    A & I \\
    I & B^{-1}
  \end{smallmatrix}\right)$
yields
\begin{align}
  \label{eq:applying}
  n_-(A) + n_-\left(B^{-1} - A^{-1}\right)
  &= n_-\left(B^{-1}\right) + n_-\left(A - I
  \left(B^{-1}\right)^{-1}\right) \\
  &= n_-(B) + n_-(A-B),
\end{align}
which establishes~\eqref{eq:crazy_formula}.  Proposition~\ref{prop:kerAkerB} can
be proved analogously; this requires a suitable extension
of the Haynsworth formula such as \cite[Thm.~6.1]{Mad_laa88} or
\cite[Thm.~A.1]{BerCanCoxMar_paa22}.
\hfill$\Diamond$\end{remark}

\bibliography{bk_bibl,additional}

\def\cprime{$'$} \def\cprime{$'$} \def\cprime{$'$} \def\cprime{$'$}
  \def\cprime{$'$} \def\cprime{$'$} \def\cprime{$'$}
  \def\polhk#1{\setbox0=\hbox{#1}{\ooalign{\hidewidth
  \lower1.5ex\hbox{`}\hidewidth\crcr\unhbox0}}} \def\cprime{$'$}
  \def\cprime{$'$}
\begin{thebibliography}{10}

\bibitem{AgrBarBes_n23}
{\sc A.~Agrachev, S.~Baranzini, and I.~Beschastnyi}, {\em Index theorems for
  graph-parametrized optimal control problems}, Nonlinearity, 36 (2023),
  pp.~2792--2838.

\bibitem{AgrGam_SymplMeth}
{\sc A.~Agrachev and R.~Gamkrelidze}, {\em Symplectic methods for optimization
  and control}, in Geometry of feedback and optimal control, vol.~207 of
  Monogr. Textbooks Pure Appl. Math., Dekker, New York, 1998, pp.~19--77.

\bibitem{AloBanBer_cmp17}
{\sc L.~Alon, R.~Band, and G.~Berkolaiko}, {\em Nodal statistics on quantum
  graphs}, Comm. Math. Phys., 362 (2018), pp.~909--948.

\bibitem{Arn_faa67}
{\sc V.~I. Arnold}, {\em On a characteristic class entering into conditions of
  quantization}, Funkcional. Anal. i Prilo\v{z}en., 1 (1967), pp.~1--14.

\bibitem{Arn_fap85}
\leavevmode\vrule height 2pt depth -1.6pt width 23pt, {\em Sturm theorems and
  symplectic geometry}, Funktsional. Anal. i Prilozhen., 19 (1985), pp.~1--10,
  95.

\bibitem{BanBerRazSmi_cmp12}
{\sc R.~Band, G.~Berkolaiko, H.~Raz, and U.~Smilansky}, {\em The number of
  nodal domains on quantum graphs as a stability index of graph partitions},
  Commun. Math. Phys., 311 (2012), pp.~815--838.

\bibitem{BanProSof_prep25}
{\sc R.~Band, M.~Prokhorova, and G.~Sofer}, {\em Spectral flow and {R}obin
  domains on metric graphs}, 2025.
\newblock preprint {\tt arXiv:2505.02039}.

\bibitem{BarOffPorWu_mz21}
{\sc V.~L. Barutello, D.~Offin, A.~Portaluri, and L.~Wu}, {\em Sturm theory
  with applications in geometry and classical mechanics}, Math. Z., 299 (2021),
  pp.~257--297.

\bibitem{Beck}
{\sc M.~Beck}, {\em Spectral stability and spatial dynamics in partial
  differential equations}, Notices Amer. Math. Soc., 67 (2020), pp.~500--507.

\bibitem{Beck_2018}
{\sc M.~Beck, G.~Cox, C.~Jones, Y.~Latushkin, K.~McQuighan, and A.~Sukhtayev},
  {\em Instability of pulses in gradient reaction--diffusion systems: a
  symplectic approach}, Philosophical Transactions of the Royal Society A:
  Mathematical, Physical and Engineering Sciences, 376 (2018), p.~20170187.

\bibitem{BJ22}
{\sc M.~Beck and J.~Jaquette}, {\em Validated spectral stability via conjugate
  points}, SIAM J. Appl. Dyn. Syst., 21 (2022), pp.~366--404.

\bibitem{BehHasDeS_boundarytriples}
{\sc J.~Behrndt, S.~Hassi, and H.~de~Snoo}, {\em Boundary value problems,
  {W}eyl functions, and differential operators}, vol.~108 of Monographs in
  Mathematics, Birkh\"{a}user/Springer, Cham, 2020.

\bibitem{Ber_cmp08}
{\sc G.~Berkolaiko}, {\em A lower bound for nodal count on discrete and metric
  graphs}, Comm. Math. Phys., 278 (2008), pp.~803--819.

\bibitem{BerCanCoxMar_paa22}
{\sc G.~Berkolaiko, Y.~Canzani, G.~Cox, and J.~L. Marzuola}, {\em A local test
  for global extrema in the dispersion relation of a periodic graph}, Pure
  Appl. Anal., 4 (2022), pp.~257--286.

\bibitem{BerCoxLatSuk_prep23}
{\sc G.~Berkolaiko, G.~Cox, Y.~Latushkin, and S.~Sukhtaiev}, {\em The
  {D}uistermaat index and eigenvalue interlacing for self-adjoint extensions of
  a symmetric operator}, J. Spectr. Theory, 16 (2026), pp.~1--49.
\newblock also {\tt arXiv:2311.06701}.

\bibitem{BooFur_tjm98}
{\sc B.~Booss-Bavnbek and K.~Furutani}, {\em The {M}aslov index: a functional
  analytical definition and the spectral flow formula}, Tokyo J. Math., 21
  (1998), pp.~1--34.

\bibitem{BooZhu_memAMS}
{\sc B.~Booss-Bavnbek and C.~Zhu}, {\em The {M}aslov index in symplectic
  {B}anach spaces}, Mem. Amer. Math. Soc., 252 (2018), pp.~x+118.

\bibitem{Bot_cpam56}
{\sc R.~Bott}, {\em On the iteration of closed geodesics and the {S}turm
  intersection theory}, Comm. Pure Appl. Math., 9 (1956), pp.~171--206.

\bibitem{CapLeeMil_cpam94}
{\sc S.~E. Cappell, R.~Lee, and E.~Y. Miller}, {\em On the {M}aslov index},
  Comm. Pure Appl. Math., 47 (1994), pp.~121--186.

\bibitem{CDB09}
{\sc F.~Chardard, F.~Dias, and T.~J. Bridges}, {\em Computing the {M}aslov
  index of solitary waves. {I}. {H}amiltonian systems on a four-dimensional
  phase space}, Phys. D, 238 (2009), pp.~1841--1867.

\bibitem{CDB11}
\leavevmode\vrule height 2pt depth -1.6pt width 23pt, {\em Computing the
  {M}aslov index of solitary waves, {P}art 2: {P}hase space with dimension
  greater than four}, Phys. D, 240 (2011), pp.~1334--1344.

\bibitem{CZ84}
{\sc C.~Conley and E.~Zehnder}, {\em Morse-type index theory for flows and
  periodic solutions for {H}amiltonian equations}, Comm. Pure Appl. Math., 37
  (1984), pp.~207--253.

\bibitem{deGosson_PrinciplesNewtonianQuantum}
{\sc M.~A. de~Gosson}, {\em The principles of {N}ewtonian and quantum
  mechanics}, Imperial College Press, London, 2001.

\bibitem{Dui_am76}
{\sc J.~J. Duistermaat}, {\em On the {M}orse index in variational calculus},
  Advances in Math., 21 (1976), pp.~173--195.

\bibitem{ElySepSim_f23}
{\sc J.~Elyseeva, P.~\v{S}epitka, and R.~{\v S}imon~Hilscher}, {\em Comparative
  index and {H}\"ormander index in finite dimension and their connections},
  Filomat, 37 (2023), pp.~5243--5257.

\bibitem{Fur_jgp04}
{\sc K.~Furutani}, {\em Fredholm-{L}agrangian-{G}rassmannian and the {M}aslov
  index}, J. Geom. Phys., 51 (2004), pp.~269--331.

\bibitem{GarPorWu_prep25}
{\sc D.~Garrisi, A.~Portaluri, and L.~Wu}, {\em Index theory for non-compact
  quantum graphs}, 2026.
\newblock preprint {\tt arXiv:2509.09749}.

\bibitem{Hatcher}
{\sc A.~Hatcher}, {\em Algebraic topology}, Cambridge University Press,
  Cambridge, 2002.

\bibitem{Hay_laa68}
{\sc E.~V. Haynsworth}, {\em Determination of the inertia of a partitioned
  {H}ermitian matrix}, Linear Algebra and Appl., 1 (1968), pp.~73--81.

\bibitem{HofKen_lmp21}
{\sc M.~Hofmann and J.~B. Kennedy}, {\em Interlacing and {F}riedlander-type
  inequalities for spectral minimal partitions of metric graphs}, Lett. Math.
  Phys., 111 (2021), pp.~Paper No. 96, 30.

\bibitem{Hor_am71}
{\sc L.~H\"{o}rmander}, {\em Fourier integral operators. {I}}, Acta Math., 127
  (1971), pp.~79--183.

\bibitem{How_jmaa21}
{\sc P.~Howard}, {\em H\"{o}rmander's index and oscillation theory}, J. Math.
  Anal. Appl., 500 (2021), pp.~Paper No. 125076, 38.

\bibitem{HLS17}
{\sc P.~Howard, Y.~Latushkin, and A.~Sukhtayev}, {\em The {M}aslov index for
  {L}agrangian pairs on {$\Bbb{R}^{2n}$}}, J. Math. Anal. Appl., 451 (2017),
  pp.~794--821.

\bibitem{HLS18}
{\sc P.~Howard, Y.~Latushkin, and A.~Sukhtayev}, {\em The {M}aslov and {M}orse
  indices for system {S}chr\"{o}dinger operators on {$\Bbb R$}}, Indiana Univ.
  Math. J., 67 (2018), pp.~1765--1815.

\bibitem{HS16}
{\sc P.~Howard and A.~Sukhtayev}, {\em The {M}aslov and {M}orse indices for
  {S}chr\"{o}dinger operators on {$[0,1]$}}, J. Differential Equations, 260
  (2016), pp.~4499--4549.

\bibitem{Howard_22}
{\sc P.~Howard and A.~Sukhtayev}, {\em Renormalized oscillation theory for
  singular linear {H}amiltonian systems}, J. Funct. Anal., 283 (2022),
  pp.~Paper No. 109525, 74.

\bibitem{HuPorWuXin_n26}
{\sc X.~Hu, A.~Portaluri, L.~Wu, and Q.~Xing}, {\em Morse index theorem for
  heteroclinic, homoclinic and halfclinic orbits of {L}agrangian systems},
  Nonlinearity, 39 (2026), pp.~Paper No. 035005, 37.

\bibitem{HuWuYan_jdde20}
{\sc X.~Hu, L.~Wu, and R.~Yang}, {\em Morse index theorem of {L}agrangian
  systems and stability of brake orbit}, J. Dynam. Differential Equations, 32
  (2020), pp.~61--84.

\bibitem{Ke58}
{\sc J.~B. Keller}, {\em Corrected {B}ohr-{S}ommerfeld quantum conditions for
  nonseparable systems}, Ann. Physics, 4 (1958), pp.~180--188.

\bibitem{Kuc_incol01}
{\sc P.~Kuchment}, {\em The mathematics of photonic crystals}, in Mathematical
  modeling in optical science, G.~Bao, L.~Cowsar, and W.~Masters, eds., vol.~22
  of Frontiers Appl. Math., SIAM, Philadelphia, PA, 2001, pp.~207--272.

\bibitem{LS20}
{\sc Y.~Latushkin and S.~Sukhtaiev}, {\em An index theorem for
  {S}chr\"{o}dinger operators on metric graphs}, in Analytic trends in
  mathematical physics, vol.~741 of Contemp. Math., Amer. Math. Soc.,
  Providence, RI, [2020] \copyright 2020, pp.~105--119.

\bibitem{Leray_LagrangianAnalysis}
{\sc J.~Leray}, {\em Lagrangian analysis and quantum mechanics}, MIT Press,
  Cambridge, Mass.-London, 1981.

\bibitem{LionVergne_WeilMaslov}
{\sc G.~Lion and M.~Vergne}, {\em The {W}eil representation, {M}aslov index and
  theta series}, vol.~6 of Progress in Mathematics, Birkh\"{a}user, Boston,
  Mass., 1980.

\bibitem{LiuZhaZha_dcds24}
{\sc P.~Liu, D.~Zhang, and Z.~Zhao}, {\em Minimal {$P$}-symmetric periodic
  solutions in semi-positive definite {H}amiltonian systems}, Discrete Contin.
  Dyn. Syst., 44 (2024), pp.~2327--2341.

\bibitem{Mad_laa88}
{\sc J.~H. Maddocks}, {\em Restricted quadratic forms, inertia theorems, and
  the {S}chur complement}, Linear Algebra Appl., 108 (1988), pp.~1--36.

\bibitem{Meinrenken_SymplecticGeometry}
{\sc E.~Meinrenken}, {\em Symplectic geometry}.
\newblock Lecture Notes, University of Toronto, 2000.

\bibitem{Piccione}
{\sc P.~Piccione and D.~V. Tausk}, {\em A student's guide to symplectic spaces,
  {G}rassmannians and {M}aslov index}, Publica\c{c}\~{o}es Matem\'{a}ticas do
  IMPA. [IMPA Mathematical Publications], Instituto de Matem\'{a}tica Pura e
  Aplicada (IMPA), Rio de Janeiro, 2008.

\bibitem{RobSal_t93}
{\sc J.~Robbin and D.~Salamon}, {\em The {M}aslov index for paths}, Topology,
  32 (1993), pp.~827--844.

\bibitem{Schmudgen_unboundedSAO}
{\sc K.~Schm\"{u}dgen}, {\em Unbounded self-adjoint operators on {H}ilbert
  space}, vol.~265 of Graduate Texts in Mathematics, Springer, Dordrecht, 2012.

\bibitem{Wal_im69}
{\sc C.~T.~C. Wall}, {\em Non-additivity of the signature}, Invent. Math., 7
  (1969), pp.~269--274.

\bibitem{ZhoWuZhu_fmc18}
{\sc Y.~Zhou, L.~Wu, and C.~Zhu}, {\em H\"{o}rmander index in
  finite-dimensional case}, Front. Math. China, 13 (2018), pp.~725--761.

\end{thebibliography}

\bibliographystyle{siam}

\end{document}